\documentclass[12pt]{article}
\usepackage[centertags]{amsmath}
\usepackage{amsfonts}
\usepackage{amssymb}
\usepackage{latexsym}
\usepackage{amsthm}
\usepackage{newlfont}
\usepackage{graphicx}
\usepackage{listings}
\usepackage{booktabs}
\usepackage{abstract}
\usepackage{enumerate}
\usepackage{xcolor}
\RequirePackage{srcltx}
\date{}

\newlength{\defbaselineskip}
\newcommand{\setlinespacing}[1]%
           {\setlength{\baselineskip}{#1 \defbaselineskip}}

\newcommand{\bR}{{\mathbb{R}}}

\newcommand{\N}{{\mathbb{N}}}

\newcommand{\actaqed}{\hfill $\actabox$}
{\medskip\noindent \textit{Proof of #1. }}%
{\actaqed \medskip}

\def\D{{\mathcal D}}

\def\cB{{\mathcal B}}
\def\C{{\mathcal C}}
\def\cC{{\mathcal C}}
\def\cD{{\mathcal D}}

\def \Tr{\mathcal T}
\def \cK{\mathcal K}

\def\cU{{\mathcal U}}
\def\cV{{\mathcal V}}
\def\cW{{\mathcal W}}

\def \V{\mathcal V}

\def\R{{\mathbb R}}
\def\Z{\mathbb Z}

\def \T{\mathbb T}
\def\bP{\mathbb P}
\def\bE{\mathbb E}
\def\bbC{\mathbb C}
\def \<{\langle}
\def\>{\rangle}

\def \La{\Lambda}
\def\Og{\Omega}

\def \e{\varepsilon}
\def \de{\delta}

\def\al{\alpha}

\def\la{\lambda}

\def \sp{\operatorname{span}}

\def\bx{\mathbf x}
\def\by{\mathbf y}
\def\bz{\mathbf z}
\def\bk{\mathbf k}

\def\bn{\mathbf n}

\newtheorem{Theorem}{Theorem}[section]
\newtheorem{Lemma}{Lemma}[section]

\newtheorem{Proposition}{Proposition}[section]
\newtheorem{Corollary}{Corollary}[section]
\theoremstyle{definition}
\newtheorem{Remark}{Remark}[section]
\newtheorem{Example}{Example}[section]
\numberwithin{equation}{section}

\newcommand{\be}{\begin{equation}}
\newcommand{\ee}{\end{equation}}

\begin{document}

\title{Sampling discretization of the uniform norm and applications}

\author{E. Kosov\thanks{Steklov Mathematical Institute of Russian Academy of Sciences, Moscow, Russia;
National Research University Higher School of Economics, Russian Federation.}\,
 and   V. Temlyakov\thanks{ Steklov Mathematical Institute of Russian Academy of Sciences, Moscow, Russia; Lomonosov Moscow State University;  Moscow Center of Fundamental and Applied Mathematics;  University of South Carolina.}}

\newcommand{\Addresses}{{
  \bigskip
  \footnotesize

\medskip
E.D. Kosov, \textsc{
Steklov Mathematical Institute of Russian Academy of Sciences, Moscow, Russia; \\
National Research University Higher School of Economics, Russian Federation.
  \\
E-mail:} \texttt{ked\_2006@mail.ru}

 \medskip
  V.N. Temlyakov, \textsc{Steklov Mathematical Institute of Russian Academy of Sciences, Moscow, Russia; \\Lomonosov Moscow State University;  \\Moscow Center of Fundamental and Applied Mathematics;  \\University of South Carolina.
  \\
E-mail:} \texttt{temlyakovv@gmail.com}

}}
\maketitle

\begin{abstract}
	{Discretization of the uniform norm of functions from a given finite dimensional subspace of continuous functions is studied. Previous known results show that for any $N$-dimensional subspace of the space of continuous functions it is sufficient to use $e^{CN}$ sample points for an accurate upper bound for the uniform norm by the discrete norm and   that one cannot improve on the exponential growth of the number of sampling points for a good discretization theorem in the uniform norm. In this paper we focus on two types of results, which allow us to obtain good discretization of the uniform norm with polynomial in $N$ number of points. In the first way we weaken the discretization inequality by allowing a bound of the uniform norm by the discrete norm multiplied by an extra factor, which may depend on $N$. In the second way we impose restrictions on the finite dimensional subspace under consideration. In particular,
we prove a general result, which connects the upper bound on the number of sampling points in the discretization theorem for the uniform norm with the best $m$-term bilinear approximation of the Dirichlet kernel associated with the given subspace.
 }
\end{abstract}


{\it MSC classification 2000:} Primary 65J05; Secondary 42A05, 65D30, 41A63.

\section{Introduction}
\label{I}

Last decade was the period of systematic study of the problem of discretization of the $L_q$, $1\le q\le\infty$, norms of elements of finite dimensional subspaces (see the survey papers \cite{DPTT} and \cite{KKLT}). It was understood that the problem of discretization of the uniform norm ($L_\infty$) is
fundamentally different from the discretization problem of the integral norms ($L_q$, $1\le q <\infty$).
In this paper we mostly focus on discretization of the uniform norm. We begin with some standard notations and definitions.

Let $\Omega$ be a subset of $\R^d$ with a
probability measure $\mu$. By the $L_q$, $1\le q< \infty$, norm we understand
$$
\|f\|_q:=\|f\|_{L_q(\Omega,\mu)} := \left(\int_\Omega |f|^qd\mu\right)^{1/q}.
$$
By the $L_\infty$ norm we understand the uniform norm on the space $\cC(\Omega)$ of all continuous functions on $\Omega$, i.e.
$$
\|f\|_\infty := \max_{\bx\in\Omega} |f(\bx)|\quad \forall f\in \cC(\Omega).
$$
In the cases when it is important to emphasize which norm is being used (e.g. for the entropy numbers in section 4),
we will write $L_\infty(\Omega)$ in place of $\cC(\Omega)$.

In this paper we mainly study the following discretization problem for the uniform norm.

{\bf The Bernstein discretization problem.} Let $C>0$ and $\Omega$ be a compact subset of $\R^d$.
For which $m$ a given subspace $X_N\subset \mathcal C(\Omega)$
(index $N$ here, usually, stands for the dimension of $X_N$)
has the property: There exists a set
$$
\Big\{\xi^j \in \Omega: j=1,\dots,m\Big\}
$$
 such that for any $f\in X_N$ we have
\be\label{BI1}
 \|f\|_\infty \le C\max_{1\le j\le m} |f(\xi^j)| ?
\ee

The first results on discretization of the uniform norm
were obtained by Bernstein \cite{Bern1} and \cite{Bern2} (see also \cite{Z}, Ch.10, Theorem (7.28)). In recent years this problem
has been extensively studied in various settings
(e.g. see \cite{Kr11},  \cite{VT168}, \cite{KKT},
   \cite{DP23}, surveys \cite{DPTT}, \cite{KKLT} and citations therein).

Along with the above problem of the uniform norm discretization
we will also consider discretization of the general $L_q$ norms
of elements of finite-dimensional subspaces.
Namely, the following problem will also be studied.

{\bf The Marcinkiewicz discretization problem.}
Let $\Omega$ be a compact subset of $\R^d$ with a probability measure $\mu$.
We say that a linear subspace $X_N$
of $L_q(\Omega,\mu)$, $1\le q < \infty$, admits the Marcinkiewicz-type discretization theorem
with parameters $m\in \N$ and $q$ and positive constants $C_1\le C_2$ if there exists a set
$$
\Big\{\xi^j \in \Omega: j=1,\dots,m\Big\}
$$
such that for any $f\in X_N$ we have
\be\label{I1a}
C_1\|f\|_q^q \le \frac{1}{m} \sum_{j=1}^m |f(\xi^j)|^q \le C_2\|f\|_q^q.
\ee
We are interested in the following problem.
For which $m$ a given subspace $X_N\subset L_q(\Omega,\mu)$ does
admit the Marcinkiewicz-type discretization theorem with parameters
$m\in \N$ and $q$?

Inequalities (\ref{I1a}) are also known in the literature under the name of Marcinkiewicz-Zygmund inequalities.

The following results on discretization of the uniform norm of elements of an $N$-dimensional subspace $X_N$ of $\cC(\Omega)$ are known.

\begin{Theorem}[\cite{KKT}]\label{IT3} Let $X_N$ be an $N$-dimensional subspace of $\cC(\Omega)$.
There exists a set $\xi=\{\xi^\nu\}_{\nu=1}^m$ of $m\le 9^N$ points such that for any $f\in X_N$ we
have
\be\label{I2a}
\|f\|_\infty \le 2\max_\nu |f(\xi^\nu)|.
\ee
\end{Theorem}

\begin{Theorem}[\cite{VT168},\cite{DPTT}]\label{IT2} Let $\Lambda_N = \{k_j\}_{j=1}^N$ be a lacunary sequence: $k_1=1$, $k_{j+1} \ge bk_j$, $b>1$, $j=1,\dots,N-1$. Assume that a finite set $\xi=\{\xi^\nu\}_{\nu=1}^m\subset \mathbb T$ has
the following property
\begin{equation}\label{I1b}
\|f\|_\infty \le L\max_{\nu}|f(\xi^\nu)| \qquad
\forall f\in \Tr(\Lambda_N)
\end{equation}
where
$\Tr(\Lambda_N):=
\bigl\{f(x)=\sum_{k\in \Lambda_N}c_je^{ikx}\bigr\}$.
Then
$$
m \ge (N/e)e^{CN/L^2}
$$
with a constant $C>0$ which may only depend on $b$.
\end{Theorem}

Thus, for discretization of the uniform norm with the bound (\ref{BI1}) it is sufficient to use exponentially (in $N$) many   points and this cannot be improved in general case. Certainly, for practical applications exponentially many points are forbidden. The main goal of this paper is to continue a discussion of the following two problems (see \cite{NoLN}, \cite{DPTT}, \cite{KKT}, \cite{KKLT}).

{\bf Problem 1.} How can we weaken the inequality
(\ref{BI1}) in order to obtain discretization with reasonable number of points?

{\bf Problem 2.} Under which conditions on $X_N$ can we obtain discretization (\ref{BI1}) with reasonable number of points?

Problem 1 is addressed in Sections \ref{A} and \ref{subDir}. Problem 2 is addressed in Section \ref{BD}.
We now give a very brief description of some results from those sections.

In Section \ref{A} we provide a sharper version of Theorem 6.6 from \cite{KKT}.
That theorem establishes a bound of type (\ref{BI1})
under the Nikol'skii's inequality $NI(2,\infty)H$ assumption (see the definition in Section \ref{A})
but with constant $C$ dependent on $N$.
In Theorem \ref{AT3}
we provide a bound for a weighted uniform norm
$\|\varphi(x)f(x)\|_\infty$ instead of the usual $\|f(x)\|_\infty$,
where the function $\varphi(x)$ is determined by the subspace $X_N$ and
actually is the Christoffel function of this subspace (see the definition in Section \ref{A}).

In Section \ref{subDir} we prove
the discretization inequality (\ref{BI1}) with an additional factor in the right hand side.
That additional factor comes from the special type of bilinear approximation of the Dirichlet kernel
associated with the subspace $X_N$. Section \ref{subDir} is a follow up of the paper \cite{KKT}.

In Section \ref{BD} we use the following general idea for studying Problem 2. For a given subspace
$X_N$ we use the discretization results of the $L_p$ norm with large $p$, which depends on $N$
(roughly $p= \ln N$). In order to apply such strategy we need to have good bounds on $m$ for discretization of the $L_p$ norm in terms of both the dimension $N$ and the parameter $p$.
There are two standard assumptions on $X_N$, which allow us to control $m$ for good discretization
of the $L_p$ norm. One assumption is in terms of the Nikol'skii inequality $NI(p,\infty)BN^{1/p}$ and
the other one is in terms of the entropy numbers of the $L_p$-unit ball of $X_N$ in the uniform norm.
It is known that these two assumptions are related -- the entropy numbers assumption is slightly stronger than the Nikol'skii inequality assumption. In Section \ref{BD} we present proofs under both assumptions.
Moreover, for the case of the entropy numbers assumption, we present two distinct proofs.
The main reason for that is that those proofs are based on different kinds of techniques and
one of the proofs under the entropy assumption seems to be more elementary than the other one.
We prove in Section \ref{BD} that under those assumptions we can guarantee discretization of the uniform norm in the form (\ref{BI1}) with the number $m$ of points growing polynomially in the dimension $N$.

In Section \ref{Re} we apply the discretization of the uniform norm results to the Remez-type inequalities. It was proved in \cite{TT} (see Theorem \ref{ReT1} below) that discretization of the uniform norm of trigonometric polynomials implies the corresponding Remez inequality. Typically, the Remez inequalities are obtained for
trigonometric polynomials with frequencies from standard domains. In the univariate case it is a segment
$[-n,n]$, in the multivariate case it is a box $[-n_1,n_1]\times\cdots\times[-n_d,n_d]$ or a hyperbolic cross $\Gamma(N)$. An important feature of the
results in Section \ref{Re} is that we prove
the Remez-type inequalities for subspaces of multivariate trigonometric polynomials with frequencies
from an arbitrary finite set $Q\subset \Z^d$.

\section{Discretization of the uniform norm under the Nikol'skii inequality assumption}
\label{A}

We firstly recall the Nikol'skii-type inequality assumption.

{\bf Nikol'skii-type inequalities.} Let $1\le q<p\le \infty$, and $X_N\subset L_q(\Omega)$. The inequality
\begin{equation}\label{I4a}
\|f\|_p \leq H\|f\|_q,\   \ \forall f\in X_N
\end{equation}
is called the Nikol'skii inequality for the pair $(q,p)$ with the constant $H$.
We will also use the brief form of this fact: $X_N \in NI(q,p)H$.

In \cite{KKT} the following observation was made.

\begin{Theorem}\label{AT2} There are two absolute constants $C_1$ and $C_2$ such that for any
$X_N \in NI(2,\infty)H$  there exists a set of $m\leq C_1N$ points
$\xi^1,\ldots, \xi^m\in\Omega$ with the property: For any $f\in X_N$ we have
\be\label{A-AT2}
\|f\|_\infty \le C_2 H \max_j |f(\xi^j)|.
\ee
\end{Theorem}


We note that the
Nikol'skii-type inequality assumption $NI(2,\infty)H$ is in some sense natural
and is valid for many subspaces (for example, for subspaces of trigonometric polynomials with
any set of frequencies).

\begin{Remark}\label{AR2}
{\rm
Let $\cU_N:=\{u_j(\bx)\}_{j=1}^N$ be an orthonormal  basis of
$X_N$ on $(\Og,\mu)$.
It is well known and easy to see that
$$
\sup_{f\in  X_N, \|f\|_2\le 1} |f(\bx)| = \sup_{\sum_{j=1}^N|a_j|^2\le 1}\Bigl|\sum_{j=1}^N a_ju_j(\bx)\Bigr|
= \left(\sum_{j=1}^N |u_j(\bx)|^2\right)^{1/2} =: w(\bx).
$$
Thus, the condition $X_N \in NI(2,\infty)H$
is equivalent to the assumption that
for an orthonormal basis $\{u_1, \ldots, u_N\}$
in $X_N$ one has
$$
|u_1(\bx)|^2+\ldots+|u_N(\bx)|^2\le H^2.
$$
In particular, one always has the bound $H\ge\sqrt{N}$.
The function $w(\bx)^{-1}$ is known as the Christoffel function of the subspace $X_N$.
}
\end{Remark}

We note that the bound in Theorem \ref{AT2} is in some sense sharp.

\begin{Remark}\label{AR3}
{\rm
It is known (see Theorem 2.3 in \cite{DPTT} and Corollary in
\cite{VT168}) that for every number $R>0$ there is a constant
$C(R)>0$ such that for any pair of numbers
$N, m \in \mathbb{N}$, $m\le RN$,
there is a trigonometric subspace
\begin{equation}\label{eq-trig}
\mathcal{T}(Q)=\Bigl\{f(\bx)=\sum_{k\in Q}c_je^{i\bx k}\colon k\in Q\Bigr\},
\quad |Q|=N,
\end{equation}
such that
$$
D(Q,m,d):=\inf_{\xi^1,\ldots, \xi^m}\sup_{f\in \mathcal{T}(Q)}
\frac{\|f\|_\infty}{\max_{1\le j\le m}|f(\xi^j)|}\ge
C(R)\sqrt{N}.
$$
Thus, Theorem \ref{AT2} could not be improved
if we impose only the assumption $X_N \in NI(2,\infty)H$.
Moreover, imposing only this condition,
we cannot substantially improve
the extra factor in (\ref{A-AT2}) even
if we allow $m$ to be of order $N^k$ with some fixed $k$.
In particular (see \cite[Proposition 1.1]{KKT}),
for any number $k\in\N$
and any constant $R\ge 1$ there exists a positive constant $C=C(k,R)$ such that
for any pair of numbers
$N, m \in \mathbb{N}$, $m\le RN^k$,
there is a trigonometric subspace
$\mathcal{T}(Q)$ with $|Q|=N$
such that
$$
D(Q,m,d)\ge
C(k, R)\Bigl(\frac{N}{\log N}\Bigr)^{1/2}.
$$
}
\end{Remark}

Theorem \ref{AT2} was deduced from the following general fact concerning the discretization
of $L^2$-norm from \cite[Theorem 1.2]{LT} (see \cite{DPSTT2} for the proof of this result in the real case).

\begin{Theorem}\label{AT1} If $X_N$ is an $N$-dimensional subspace of the complex $L_2(\Omega,\mu)$, then there exist three absolute positive constants $C_1'$, $c_0'$, $C_0'$,  a set of $m\leq   C_1'N$ points $\xi^1,\ldots, \xi^m\in\Omega$, and a set of nonnegative  weights $\lambda_j$, $j=1,\ldots, m$,  such that
\[ c_0'\|f\|_2^2\leq  \sum_{j=1}^m \lambda_j |f(\xi^j)|^2 \leq  C_0' \|f\|_2^2,\  \ \forall f\in X_N.\]
\end{Theorem}

For our application we need an extra property of weights. We use the following Remark \ref{AR1} from \cite{VT183}.
\begin{Remark}\label{AR1} Considering a new subspace $X_N' := \{f\,:\, f= g+c, \, g\in X_N,\, c\in \bbC\}$ and applying Theorem \ref{AT1} to
the $X_N'$ with $f=1$ ($g=0$, $c=1$) we conclude that a version of Theorem \ref{AT1} holds with $m\le C_1'N$ replaced by $m\le C_1'(N+1)$ and with weights satisfying
$$
\sum_{j=1}^m \lambda_j \le C_0'.
$$
\end{Remark}

We now prove a weighted generalization of Theorem \ref{AT2}.

\begin{Theorem}\label{AT3}
There are two absolute constants $C_1$ and $C_2$ such that for any $X_N$
there exists a set of $m\leq C_1N$ points $\xi^1,\ldots, \xi^m\in\Omega$ with the property:
For any $f\in X_N$ we have
$$
\|(|u_1|^2+\ldots+|u_N|^2)^{-1/2}f\|_\infty \le C_2 \max_j |f(\xi^j)|,
$$
where $\{u_1, \ldots, u_N\}$ is any orthonormal basis in $X_N$ with respect to any probability
measure $\mu$ on $\Omega$.
In particular,
$$
\|f\|_\infty \le C_2
\|w\|_\infty \cdot \max_j |f(\xi^j)|,
$$
where
$w(\bx):=(|u_1(\bx)|^2+\ldots+|u_N(\bx)|^2)^{1/2}$.
\end{Theorem}
\begin{proof} Without loss of generality we assume that
$1\in X_N$.
Let $\{u_1, \ldots, u_N\}$ be an orthonormal basis of $X_N$.
We set
$$
U(\bx):=\frac{1}{N}\sum_{k=1}^{N}|u_k(\bx)|^2,\quad \widetilde{u}_j:= U^{-1/2}u_j,\quad
\widetilde{\mu}:= U\cdot\mu
$$
and we consider this new probability measure $\widetilde{\mu}$ on
$\widetilde{\Omega}:=\{\bx\in \Omega\colon U(\bx)>0\}$.
Let $\widetilde{X}_N:=span\{\widetilde{u}_1, \ldots, \widetilde{u}_N\}$. Then
$\{\widetilde{u}_1, \ldots, \widetilde{u}_N\}$ is an orthonormal basis in $\widetilde{X}_N$
with respect to the standard inner product in $L^2(\widetilde{\mu})$. Moreover,
$$
|\widetilde{u}_1(\bx)|^2 + \ldots + |\widetilde{u}_N(\bx)|^2 = N\quad \forall
\bx\in \widetilde{\Omega},
$$
i.e. $\widetilde{X}_N \in NI(2,\infty)\sqrt{N}$.

By Theorem \ref{AT1}
there are $m\le C_1N$ points $\xi^1,\ldots, \xi^m\in\Omega$ such that
$$
C_2\|\widetilde{f}\|_{L^2(\widetilde{\mu})}^2\le
\sum\limits_{j=1}^m\lambda_j|\widetilde{f}(\xi^j)|^2\le C_3\|\widetilde{f}\|_{L^2(\widetilde{\mu})}^2
\quad \forall \widetilde{f}\in \widetilde{X}_N.
$$
We note that, for each $f\in X_N$, the function
$\widetilde{f}=U^{-1/2}f$ belongs to the space~$\widetilde{X}_N$ and that
$\|\widetilde{f}\|_{L^2(\widetilde{\mu})} = \|f\|_{L^2(\mu)}$.
Thus, for every $f\in X_N$
\begin{multline*}
\|U^{-1/2}f\|_\infty^2\le N\|U^{-1/2}f\|_{L^2(\widetilde{\mu})}^2
\le \frac{1}{C_2} N
\sum\limits_{j=1}^m\lambda_jU^{-1}(\xi_j)|f(\xi^j)|^2\le\\ \le
\max_j |f(\xi^j)|^2
\frac{1}{C_2}N \sum\limits_{j=1}^m\lambda_jU^{-1}(\xi_j).
\end{multline*}
On the other hand, taking $f = 1$ we get that
$$
 \sum\limits_{j=1}^m\lambda_jU^{-1}(\xi_j)\le
C_3\|1\|_{L^2(\mu)}^2 = C_3.
$$
Summing up all the above estimates we get the bound
\begin{multline*}
\sup\limits_{\bx\in \Omega}\Bigl|\frac{f(\bx)}{(|u_1(\bx)|^2+\ldots+|u_N(\bx)|^2)^{1/2}}\Bigr|
= \frac{1}{\sqrt{N}}\|U^{-1/2}f\|_\infty
\le \\ \le
\frac{1}{\sqrt{N}}\max_j |f(\xi^j)|
\sqrt{\frac{C_3N}{C_2}}
= \sqrt{\frac{C_3}{C_2}}\max_j |f(\xi^j)|.
\end{multline*}
The theorem is proved.
\end{proof}

David Krieg, while working on the preprints \cite{KPUU23} and \cite{KPUU24}, pointed out to the authors that there is a known result of J. Kiefer and J. Wolfowitz \cite{KW}
(see also \cite[Section 2]{KPUU24}),
which guarantees that for any finite dimensional real subspace $X_N$ of $\cC(\Omega)$
there exists a probability measure $\mu$ on $\Omega$ such that for all
$f\in X_N$ we have

\be\label{KW}
\|f\|_\infty \le \sqrt{N}\|f\|_{L_2(\Omega,\mu)}.
\ee

This and Theorem \ref{AT3} or Theorem \ref{AT2} imply the following statement.

\begin{Corollary}\label{Cor2.1}
There are two absolute constants $C_1$ and $C_2$ such that for any real $X_N \subset \cC(\Omega)$
there exists a set of $m\leq C_1N$ points $\xi^1,\ldots, \xi^m\in\Omega$ with the property:
For any $f\in X_N$ we have
\begin{equation}\label{eq-bound-0}
\|f\|_\infty \le C_2\sqrt{N} \max_j |f(\xi^j)|.
\end{equation}
\end{Corollary}

Thus, in the case of real subspaces Corollary \ref{Cor2.1} is a direct corollary of known results (Theorem  \ref{AT2} and \cite{KW}). A standard simple argument derives from \eqref{KW} a similar inequality in the complex case with $\sqrt{N}$ replaced by $2\sqrt{N}$, which makes Corollary \ref{Cor2.1} to hold in the complex case as well. The authors of \cite{KPUU24} proved the exact analog of \eqref{KW} in the complex case, which required a non-trivial argument, and also generalized the result to the case of bounded functions instead of continuous ones. They also provided a stronger and sharper discretization result, compared to the simple combination of two known theorems in the form of Corollary \ref{Cor2.1}. In particular, they managed to prove the discretization inequality of the form \eqref{eq-bound-0} with $2N$ points in the case of bounded functions (instead of only continuous ones).

\begin{Remark}\label{AR4}
{\rm
If we use Corollary 2.5 from \cite{Kos22}
instead of Theorem \ref{AT1}, we get the following analog with $\varepsilon$ of the above theorem.
There is an absolute constant $C$ such that for any $\varepsilon\in(0, 1)$ and any
$N$-dimensional subspace $X_N$ of the complex $L_2(\Omega,\mu)$
there exists a set of $m\leq C\varepsilon^{-2} N$ points
$\xi^1,\ldots, \xi^m\in\Omega$ with the property:
For any $f\in X_N$ we have
$$
\|(|u_1|^2+\ldots+|u_N|^2)^{-1/2}f\|_\infty \le (1+\varepsilon)\max_j |f(\xi^j)|,
$$
where $\{u_1, \ldots, u_N\}$ is any orthonormal basis in $X_N$.
}
\end{Remark}

\section{Discretization and bilinear approximations}
\label{subDir}
It was pointed out in \cite{KKT} that sampling discretization of the uniform norm of functions from $X_N$ is connected with a special type of bilinear approximation of the Dirichlet kernel $\cD(X_N,\bx,\by)$ of this subspace. In this subsection we recall those results and
compare them with Theorems \ref{AT2} and \ref{AT3}.
For convenience we only discuss here the case of real functions.
To formulate Theorem \ref{LiT1} below we need to make some
preparations.

For an orthonormal system $\cU_N:=\{u_j(\bx)\}_{j=1}^N$  on $(\Og,\mu)$
we define the {\bf Dirichlet kernel} as follows
$$
\cD_N(\cU_N,\bx,\by) := \sum_{j=1}^N u_j(\bx)u_j(\by).
$$

It is known that the Dirichlet kernel $\cD_N(\cU_N,\bx,\by)$ does not depend on the orthonormal basis of a given subspace $X_N$, i.e.
for any two orthonormal bases $\cU_N$ and $\cV_N$ of
a given subspace $X_N$ we have
$$
\cD_N(\cU_N,\bx,\by) = \cD_N(\cV_N,\bx,\by).
$$
This observation shows that the Dirichlet kernel $\cD_N(\cU_N,\bx,\by)$ with $\cU_N$ being an orthonormal basis
of $X_N$ can be seen as a characteristic of the subspace $X_N$.
Denote
$$
\cD(X_N,\bx,\by) := \cD_N(\cU_N,\bx,\by).
$$

Consider the following problem of constrained best $M$-term approximation with respect to the bilinear dictionary. Define
\begin{align*}
\cB_M(X_N^\perp)&:= \Big\{\cW\,:\, \cW(\bx,\by)= \sum_{i=1}^M w_i(\bx)v_i(\by),\, w_i\in \cC,\,v_i\in L_1, i=1,\dots,M,\\
&\text{satisfying the condition: For any $f\in X_N$ and each $\bx\in \Omega$ we have}\\
&\int_{\Omega} \cW(\bx,\by) f(\by)d\mu =0\Big\}.
\end{align*}
For a function $\cK(\bx,\by)$ continuous on $\Og\times\Og$ consider
$$
\sigma_M^c(\cK)_{(\infty,1)}
:= \inf_{\cW\in \cB_M(X_N^\perp)} \sup_{\bx\in \Omega}\|\cK(\bx,\cdot) - \cW(\bx,\cdot)\|_1.
$$
Let now $\cW_M=\sum_{i=1}^M w_i^*v_i^*\in \cB_M(X_N^\perp)$ be such that
$$
\sup_{\bx\in \Omega}\|\D(X_N)(\bx,\cdot) - \cW_M(\bx,\cdot)\|_1 \le 2\sigma_M^c(\D(X_N))_{(\infty,1)} .
$$
Assume as above that $\{u_j\}_{j=1}^N$ is an orthonormal basis   with respect to the measure $\mu$ of the subspace $X_N$. Consider a subspace of $L_1(\Omega,\mu)$
$$
Y_S :=\sp \{v_j^*(\by)f(\by),  u_i(\by)f(\by)\,:\, f\in X_N, \, j=1,\dots,M, \, i=1,\dots,N\}.
$$
Then $S:=\dim Y_S \le (M+N)N$.

{\bf Condition D.} Suppose that $X_N$ is such that there exists a set of points $\{\xi^\nu\}_{\nu=1}^m$ and a set of positive weights $\{\la_\nu\}_{\nu=1}^m$ such that for any $g\in Y_S$
\be\label{df4}
\frac{1}{2} \|g\|_1 \le \sum_{\nu=1}^m \la_\nu |g(\xi^\nu)| \le \frac{3}{2}\|g\|_1 .
\ee

\begin{Theorem}[{\cite[Theorem 2.1]{KKT}}]\label{LiT1}  Let $X_N\subset \cC(\Omega)$ be an $N$-dimensional subspace. Assume that function $1$ belongs to $X_N$ (if not, we include it, which results in increase of dimension by $1$). Assume that $X_N$ satisfies Condition D. Then for the set of points $\{\xi^\nu\}_{\nu=1}^m$ from Condition D we have: For any $f\in X_N$
	$$
	\|f\|_\infty \le 6\sigma_M^c(\D(X_N))_{(\infty,1)}\left(\max_{\nu}|f(\xi^\nu)|\right).
	$$
\end{Theorem}

\begin{Remark}[{\cite[Remark 2.1]{KKT}}]\label{LiR1} We can modify Condition D by replacing constants $1/2$ and $3/2$ in (\ref{df4}) by positive constants $c_0$ and $C_0$. Then Theorem \ref{LiT1} holds with the constant $6$ replaced by $C(c_0,C_0)$.
\end{Remark}

We note that Theorem \ref{LiT1} is a conditional result, which guarantees good discretization of the uniform norm
under a certain condition on the subspace $X_N$.
However, it provides a rather powerful tool to study
the uniform norm discretization.
We now combine Theorem \ref{LiT1} with the following result
from \cite{DKT}
(for an earlier result in this
direction see Theorem 2.3 \cite{DPSTT2}).

\begin{Theorem}[{\cite[Corollary 1.1]{DKT}}]\label{LiT2}
		There exists a positive absolute constant  $C$  such that for any subspace $X_N$ of $\cC(\Og)$  of dimension at most $N$,
for any  $\e \in (0,1)$ and any probability measure $\mu$,
		there is a finite set of points  $\{\xi^1,\dots,\xi^m\}\subset \Og$  with
		$$
		m\leq  C\e^{-2} N\log N,
		$$
and there is a set of nonnegative weights $\{\lambda_j\}_{j=1}^m$
		which provide the following discretization inequality for any $f\in X_N$
		\be\label{Li1}
	{(1-\e)\|f\|_{L_1(\Og,\mu)} \leq  \sum_{j=1}^m \lambda_j|f(\xi^j)| \leq  (1+\e)\|f\|_{L_1(\Og,\mu)}.}
		\ee
\end{Theorem}

Combination of Theorems \ref{LiT1} and \ref{LiT2} implies the following unconditional result.

\begin{Theorem}\label{LiT3}  Let $X_N\subset \cC(\Omega)$ be an $N$-dimensional subspace, which contains function $1$. There are two absolute constants $C_1$, $C_2$ such that for any nonnegative integer $M$   there exists a   set of $m$ points $\{\xi^\nu\}_{\nu=1}^m$, $m\le C_1 (M+N)N\log ((M+N)N)$, with the property: For any $f\in X_N$
	$$
	\|f\|_\infty \le C_2\sigma_M^c(\D(X_N))_{(\infty,1)}\left(\max_{\nu}|f(\xi^\nu)|\right)  .
	$$
\end{Theorem}
\begin{proof}
For the subspace $Y_S$ defined above the Condition D we apply Theorem \ref{LiT2} with $\e=1/2$.
It provides the set of points $\{\xi^\nu\}_{\nu=1}^m$ satisfying inequalities (\ref{df4}) of the Condition D with
$m \le 4CS\log S$, where $C$ is from Theorem \ref{LiT2}. Taking into account the fact that $S\le (M+N)N$ and applying Theorem \ref{LiT1} we complete the proof.
\end{proof}

\begin{Remark} In \cite{KKT} we studied a problem closely connected with the problem of
estimation of the quantities  $\sigma_M^c(\D_{X_N})_{(\infty,1)}$ in a special case of trigonometric polynomials. In this case $X_N$ is a subspace $\Tr(Q)$ of trigonometric polynomials with frequencies from $Q\subset \Z^d$, $|Q|=N$. Then $\Omega = [0,2\pi]^d$, $\mu$ is the normalized Lebesgue measure, and for $\bx,\by\in [0,2\pi]^d$ we have

$$
\D_{\Tr(Q)}(\bx,\by) = \D_Q(\bx-\by),\qquad    \D_Q(\bx):= \sum_{\bk\in Q} e^{i(\bk,\bx)}.
$$

Clearly, for any  $g\in \Tr(\La)$, $Q\cap \Lambda =\emptyset$, $|\Lambda|=M$, we have

$$
\sigma_M^c(\D_{\Tr(Q)})_{(\infty,1)} \le \|\D_Q -g\|_1.
$$

This motivated us to construct generalized de la Vall{\'e}e Poussin kernels for sets $Q$. Namely, for a given $Q\subset \Z^d$ and $M\in\N$ we were interested in construction of an $M$-term trigonometric polynomial $\V_{Q,M}$ such that $\hat \V_{Q,M}(\bk) =1$ for $\bk \in Q$ with small $L_1$ norm. In particular, we proved in \cite{KKT}  that for any $Q\subset \Z^d$ there exists $\V_{Q,M}$ such that $\|\V_{Q,M}\|_1 \le 2$ 
provided $M\ge 2^{4|Q|}$. Also, we studied there the following question. Find necessary and sufficient conditions on $M$, which guarantee existence of $\V_{Q,M}$ with the property $\|\V_{Q,M}\|_1 \le C_1|Q|^\alpha$, $\alpha \in [0,1/2)$, for all $Q$ of cardinality $n$. These conditions are given in \cite{KKT}. 
Roughly, they state that $\log M$ should be of order $n^{1-2\alpha}$.
\end{Remark}

At the end of this section we show that the quantity
$\sigma_M^c(\cK)_{(\infty,1)}$ can be calculated in a slightly
different manner.
We consider another quantity
$$
\sigma_M^\perp(\cK)_{(\infty,1)}:= \inf_{\substack{\{v_i\}, \{w_i\}, v_i\in X_N^\perp, w_i\in \cC(\Og),\\ i=1,\dots,M}}\,\, \sup_{\bx\in\Og} \|\cK(\bx,\cdot)- \sum_{i=1}^M w_i(\bx)v_i(\cdot)\|_1.
$$

\begin{Proposition}\label{dfP2} We have for $\cK\in \cC(\Og\times\Og)$
	$$
	\sigma_M^c:=\sigma_M^c(\cK)_{(\infty,1)}=\sigma_M^\perp(\cK)_{(\infty,1)}=:\sigma_M^\perp.
	$$
\end{Proposition}
\begin{proof} For any $\{v_i\}, \{w_i\}, v_i\in X_N^\perp\cap L_1(\Og,\mu), w_i\in \cC(\Og), i=1,\dots,M$ we obtain
	that
	$$
	\sum_{i=1}^M w_i(\bx)v_i(\by) \in \cB(X_N^\perp).
	$$
	Therefore, $\sigma_M^c \le \sigma_M^\perp$. We now prove the inverse inequality.
	Let
	$$
	\cW= \sum_{i=1}^M w_i(\bx)v_i(\by) \in \cB(X_N^\perp).
	$$
	Our assumption that $X_N\subset \cC(\Og)$ implies that we can use an orthogonal projection of   $v_i\in L_1$ onto $X_N$
	$$
	P_{X_N}(v_i) := \sum_{j=1}^N \<v_i,u_j\>u_j,\quad h_i := v_i - P_{X_N}(v_i),\,\, i=1,\dots,M,
	$$
	for an orthonormal basis $\{u_i\}_{i=1}^N$ of $X_N$.
	Then
	$$
	\cW(\bx,\by) = \sum_{i=1}^N b_i(\bx)u_i(\by) + \sum_{i=1}^{M} {c_i}(\bx)h_i(\by),\quad b_i,c_i \in \cC(\Og).
	$$
	Our assumption $\cW \in \cB(X_N^\perp)$ implies that $b_i(\bx)=0$ for all $\bx\in\Og$ and $i=1,\dots,N$. Therefore,
	$$
	\cW(\bx,\by) =   \sum_{i=1}^{M} c_i(\bx)h_i(\by),\quad h_i \in X_N^\perp,\quad i=1,\dots,M,
	$$
	which implies $\sigma_M^\perp \le \sigma_M^c$.
\end{proof}

\section{Some results on the Bernstein discretization problem}
\label{BD}

In Section \ref{A} we proved some discretization results in a weaker than inequality (\ref{BI1}) form. This
allowed us to obtain optimal in the sense of order bounds on $m$, namely, $m\le CN$. In this section we
would like to obtain discretization in the form of (\ref{BI1}).
It has already been discussed
that the natural Nikol'skii inequality $NI(2,\infty)BN^{1/2}$ does not imply a good discretization result -- minimal $m$ may grow exponentially
in dimension $N$ (see Theorem \ref{IT2} in Section~\ref{I}). In this section we impose
more restrictive assumptions in terms of entropy numbers and Nikol'skii-type inequality.
Our goal in this section is to study how discretization results depend on the parameter $q$ of integrability
and then to obtain Bernstein-type discretization choosing suitable parameter $q$.

{\bf 1a. Entropy assumption. Introduction.}

Denote
$$
X^q_N := \{f\in X_N:\, \|f\|_q \le 1\}.
$$
We also recall the definitions of the covering numbers and the entropy numbers.
Let $X$ be a Banach space and let $B_X$ denote the unit ball of $X$ with the center at $0$. Denote by $B_X(y,r)$ a ball with center $y$ and radius $r$: $\{x\in X:\|x-y\|\le r\}$. For a compact set $A$ and a positive number $\e$ we define the covering number $N_\e(A)$
as follows
$$
N_\e(A) := N_\e(A,X)
:=\min \{n : \exists y^1,\dots,y^n, y^j\in A :A\subseteq \cup_{j=1}^n B_X(y^j,\e)\}.
$$
It is convenient to consider along with the entropy $H_\e(A,X):= \log_2 N_\e(A,X)$ the entropy numbers $\e_k(A,X)$:
$$
\e_k(A,X)  :=\inf \{\e : \exists y^1,\dots ,y^{2^k} \in A : A \subseteq \cup_{j=1}
^{2^k} B_X(y^j,\e)\}.
$$
In our definition of $N_\e(A)$ and $\e_k(A,X)$ we require $y^j\in A$. In a standard definition of $N_\e(A)$ and $\e_k(A,X)$ this restriction is not imposed.
However, it is well known (see \cite{VTbook}, p.208) that these characteristics may differ at most by a factor $2$.

We prove below in Subsection {\bf 1b} the following conditional result.
\begin{Theorem}\label{T2.1} Let $1<q<\infty$. Suppose that a subspace $X_N$ satisfies the condition
$$
\e_k(X^q_N,L_\infty) \le  B\left\{\begin{array}{ll}  (N/k)^{1/q}, &\quad k\le N,\\
 2^{-k/N},&\quad k\ge N,\end{array} \right.
$$
with $B$ satisfying the condition $B\ge 1$.
Then for large enough absolute constant $C$ there exists a set of
$$
m \le Cq4^qN^{2-1/q}B^{q}(\log_2(2N\log (4(4B)^qNq)))^2
$$
 points $\xi^j\in \Omega$, $j=1,\dots,m$,   such that for any $f\in X_N$
we have
$$
\frac{1}{2}\|f\|_q^q \le \frac{1}{m}\sum_{j=1}^m |f(\xi^j)|^q \le \frac{3}{2}\|f\|_q^q.
$$
\end{Theorem}

We are interested in applications of Theorem \ref{T2.1} for large $q$.
In the case $q=1$ Theorem \ref{T2.1} is proved in \cite{VT159}. Our proof of Theorem \ref{T2.1} goes along the lines of the proof from \cite{VT159}.
Note that it is known (see, for instance, \cite{DPSTT2}) that the entropy condition
$\e_1(X^q_N,L_\infty) \le BN^{1/q}$ implies the Nikol'skii inequality $NI(q,\infty)4BN^{1/q}$.
Thus, results of Subsection {\bf 2} below about discretization under the Nikol'skii inequality
assumption imply the corresponding corollaries for discretization under
the entropy assumption. However, more restrictive entropy condition
will result in slightly better bounds for the number of points $m$, sufficient for a good discretization. Thus,
we present in this subsection a direct proof of discretization results under
the entropy assumption. This proof is based on a straightforward application
of the chaining argument.
Here is a direct corollary of Theorem \ref{T2.1}.

\begin{Corollary}\label{BDC1} Under conditions of Theorem \ref{T2.1} there exist two positive constants
$c(B)$ and $C(B)$ such that there are $m\le N^{c(B)}$ points $\xi^1,\dots,\xi^m$ such that
\be\label{BD1}
\|f\|_\infty\le C(B)\max\limits_{1\le \nu\le m}|f(\xi^\nu)|\quad \forall f\in X_N.
\ee
\end{Corollary}

The following  Lemma \ref{AL1}  is from \cite{BLM}.

\begin{Lemma}\label{AL1} Let $\{g_j\}_{j=1}^m$ be independent random variables with $\bE g_j=0$, $j=1,\dots,m$, which satisfy
$$
\|g_j\|_1\le 2,\qquad \|g_j\|_\infty \le M,\qquad j=1,\dots,m.
$$
Then for any $\eta \in (0,1)$ we have the following bound on the probability
$$
\bP\left\{\left|\sum_{j=1}^m g_j\right|\ge m\eta\right\} < 2\exp\left(-\frac{m\eta^2}{8M}\right).
$$
\end{Lemma}

We consider measurable functions $f(\bx)$, $\bx\in \Omega$. Let $\mu$ be a probability measure on $\Omega$. Denote $\mu^m := \mu\times\cdots\times\mu$ the probability measure on $\Omega^m := \Omega\times\cdots\times\Omega$.
  For $1\le q<\infty$ define
$$
L^q_\bz(f) := \frac{1}{m}\sum_{j=1}^m |f(\bx^j)|^q -\|f\|_q^q,\qquad \bz:= (\bx^1,\dots,\bx^m).
$$

We need the following inequality (see Proposition 3.1 from \cite{VT158}).
\begin{Proposition}\label{IP1} Let $f_i\in L_q(\Omega)$, $1<q<\infty$, be such that
$$
\|f_i\|_q^q \le 1/2,\quad \|f_i\|_\infty \le M_q, \quad i=1,2;\qquad \|f_1-f_2\|_\infty \le \delta.
$$
Then
\be\label{I1}
\mu^m\{\bz: |L^q_\bz(f_1) -L^q_\bz(f_2)| \ge \eta\} < 2\exp\left(-\frac{m\eta^2}{16qM_q^{q-1}\delta}\right).
\ee
\end{Proposition}
\begin{proof} Consider the function
$$
g(\bx) := |f_1(\bx)|^q-\|f_1\|_q^q - (|f_2(\bx)|^q-\|f_2\|_q^q).
$$
Then $\int g(\bx)d\mu =0$ and our assumption $\|f_i\|_q^q \le 1/2$, $i=1,2$, implies
$$
\|g\|_1 \le 2\|f_1\|_q^q +2\|f_2\|_q^q \le 2.
$$
In order to bound $ \|g\|_\infty$ we now prove a simple Lemma \ref{IL1}.
\begin{Lemma}\label{IL1} Suppose that $\|f_i\|_\infty \le M$, $i=1,2$. Then for $1<q<\infty$
\be\label{I2}
||f_1(\bx)|^q - |f_2(\bx)|^q| \le qM^{q-1}\|f_1-f_2\|_\infty
\ee
and
\be\label{I3}
|\|f_1\|_q^q -\|f_2\|_q^q| \le qM^{q-1}\|f_1-f_2\|_\infty.
\ee
\end{Lemma}
\begin{proof} We use the simple inequality for $b\ge a \ge 0$
\be\label{I4}
b^q-a^q\le qb^{q-1}(b-a).
\ee
Inequality (\ref{I4}) directly implies (\ref{I2}). For (\ref{I3}) we have
$$
|\|f_1\|_q^q -\|f_2\|_q^q| \le qM^{q-1}|\|f_1\|_q-\|f_2\|_q|
$$
$$
\le qM^{q-1}\|f_1- f_2\|_q \le qM^{q-1}\|f_1-f_2\|_\infty.
$$
\end{proof}

By Lemma \ref{IL1} we obtain
$$
 \|g\|_\infty \le 2qM^{q-1}_q\delta.
$$
Consider $m$ independent variables $\bx^j\in \Omega$, $j=1,\dots,m$. For $\bz\in \Omega^m$ define $m$ independent random variables $g_j(\bz)$ as $g_j(\bz) := g(\bx^j)$. Clearly,
$$
\frac{1}{m}\sum_{j=1}^m g_j(\bz) = L^q_\bz(f_1)-L^q_\bz(f_2).
$$
Applying Lemma \ref{AL1} with $M=2qM^{q-1}_q\delta$ to $\{g_j\}$ we obtain (\ref{I1}).
\end{proof}

{\bf 1b. Entropy assumption. Conditional theorem for discretization in $L_q$.}

In this subsection we give a proof of Theorem \ref{T2.1}.

{\bf Proof of Theorem \ref{T2.1}.}

As we have already mentioned above, our assumptions on the entropy numbers $\e_k(X^q_N,L_\infty)$ imply
the Nikol'skii type inequality (see, for instance, \cite{DPSTT2}): for any $f\in X_N$
\be\label{2.3}
\|f\|_\infty \le 4BN^{1/q}\|f\|_q.
\ee

We consider the case $X$ is
$\C(\Omega)$ the space of functions continuous on a compact subset $\Omega$ of $\bR^d$ with the norm
$$
\|f\|_\infty:= \sup_{\bx\in \Omega}|f(\bx)|.
$$
 We write $L_\infty$ instead of $\C$ for the space of continuous functions. We use the abbreviated notations
$$
  \e_n(W):= \e_n(W,\C).
$$
In our case
\be\label{2.5}
W:= \{t\in X_N: \|t\|_q^q = 1/2\}.
\ee
Denote
$$
\e_k:=  \frac{B}{2^{1/q}}\left\{\begin{array}{ll}  (N/k)^{1/q}, &\quad k\le N,\\
 2^{-k/N},&\quad k\ge N.\end{array} \right.
$$

Our proof uses the chaining technique and goes along the lines of the proof of Theorem \ref{T2.1} from \cite{VT158}.
Inequality (\ref{2.3}) and Lemma \ref{IL1} imply for any $f,g\in X_N^q$
\be\label{2.6}
 |L_\bz^q(f)-L_\bz^q(g)| \le S_1\|f-g\|_\infty,\qquad S_1 := 2q(4BN^{1/q})^{q-1}.
\ee

  Specify $\eta=1/4$.
Denote $\de_j := \e_{2^j}$, $j=0,1,\dots$, and consider minimal $\de_j$-nets ${\mathcal N}_j \subset W$ of $W$ in $\C(\Omega)$. We use the notation $N_j:= |{\mathcal N}_j|$. Let $J:=J(N,B,q)$ be the minimal $j$ satisfying $\de_j \le 1/(8S_1)$. For $j=1,\dots,J$ we define a mapping $A_j$ that associates with a function $f\in W$ a function $A_j(f) \in {\mathcal N}_j$ closest to $f$ in the $\C$ norm. Then, clearly,
$$
\|f-A_j(f)\|_\infty \le \de_j.
$$
We use the mappings $A_j$, $j=1,\dots, J$ to associate with a function $f\in W$ a sequence (a chain) of functions $f_J, f_{J-1},\dots, f_1$ in the following way
$$
f_J := A_J(f),\quad f_j:=A_j(f_{j+1}),\quad j=J-1,\dots,1.
$$
Let us find an upper bound for $J$, defined above. Our assumption that $B\ge 1$ and definition of $J$ imply that $2^J\ge N$ and
\be\label{2.8}
B2^{-2^{J-1}/N} \ge 1/(8S_1).
\ee
We derive from (\ref{2.8})  and (\ref{2.6})
\be\label{2.9}
2^J \le 2N\log (8BS_1),\quad J \le \log(2N\log (4(4B)^qNq)) .
\ee

Set
$$
\eta_j := \frac{1}{8J},\quad j=1,\dots,J.
$$

We now proceed to the estimate of $\mu^m\{\bz:\sup_{f\in W}|L^q_\bz(f)|\ge 1/4\}$. First of all   by (\ref{2.6})  the assumption $\de_J\le 1/(8S_1)$ implies that if $|L^q_\bz(f)| \ge 1/4$ then $|L^q_\bz(f_J)|\ge 1/8$.
Rewriting
$$
L^q_\bz(f_J) = L^q_\bz(f_J)-L^q_\bz(f_{J-1}) +\dots+L^1_\bz(f_{2})-L^1_\bz(f_1)+L^1_\bz(f_1)
$$
we conclude that if $|L^q_\bz(f)| \ge 1/4$ then at least one of the following events occurs:
$$
|L^q_\bz(f_j)-L^q_\bz(f_{j-1})|\ge \eta_j\quad\text{for some}\quad j\in (1,J] \quad\text{or}\quad |L^q_\bz(f_1)|\ge \eta_1.
$$
Therefore
\begin{eqnarray}\label{2.10}
\mu^m\{\bz:\sup_{f\in W}|L^q_\bz(f)|\ge1/4\}
\le \mu^m\{\bz:\sup_{f\in {\mathcal N}_1}|L^q_\bz(f)|\ge\eta_1\} \nonumber \\
+\sum_{j\in(1,J]}\sum_{f\in {\mathcal N}_j}\mu^m
\{\bz:|L^q_\bz(f)-L^q_\bz(A_{j-1}(f))|\ge\eta_j\}\nonumber\\
\le \mu^m\{\bz:\sup_{f\in {\mathcal N}_1}|L^q_\bz(f)|\ge\eta_1\}\nonumber\\
+\sum_{j\in(1,J]} N_j\sup_{f\in W}\mu^m
\{\bz:|L^q_\bz(f)-L^q_\bz(A_{j-1}(f))|\ge\eta_j\}.
\end{eqnarray}
  Applying  Proposition \ref{IP1} we obtain
$$
\sup_{f\in W} \mu^m\{\bz:|L^q_\bz(f)-L^q_\bz(A_{j-1}(f))|\ge \eta_j\} \le 2\exp\left(-\frac{m\eta_j^2}{16qM_q^{q-1}\de_{j-1}}\right).
$$
By Nikol'skii inequality (\ref{2.3}), taking into account the definition (\ref{2.5}) of $W$, we can set
$$
M_q= 4BN^{1/q}.
$$

We now make further estimates for a specific $m\ge C_1(q)N^{2-1/q}B^{q}J^2$ with large
enough $C_1(q)$. For $j$ such that $2^{j-1}\le N$ we obtain from the definition of
$\delta_j$ that $\delta_{j-1}= \e_{2^{j-1}} = (B2^{-1/q})(N2^{-j+1})^{1/q}$ and
\be\label{q1}
\frac{m\eta_j^2}{16qM_q^{q-1}\delta_{j-1}} \ge \frac{C_1(q)N^{2-1/q}B^{q}J^22^{(j-1)/q}}{q4^{q+5}J^2B^{q}N} = \frac{C_1(q)N^{1-1/q}2^{(j-1)/q}}{q4^{q+5}} .
\ee
Taking into account that $2^{j-1}\le N$, we get $N^{1-1/q}2^{(j-1)/q} \ge 2^{j-1}$. Therefore, for $C_1(q) \ge q4^{q+6+c/2}$ we continue (\ref{q1})
$$
\ge  2^{j+1+c}.
$$
By our choice of $\delta_j=\e_{2^j}$ we get $N_j\le 2^{2^j} <e^{2^j}$ and, therefore,
\be\label{2.11}
2N_j\exp\left(-\frac{m\eta_j^2}{16qM_q^{q-1}\de_{j-1}}\right)\le 2\exp(-2^{j+c})
\ee
for   $C_1(q)\ge q4^{q+6+c/2}$.

In the case $2^{j-1}\in (N, 2^J]$ we have $\delta_{j-1}= \e_{2^{j-1}} = (B2^{-1/q})2^{-2^{j-1}/N} $ and
\be\label{q2}
\frac{m\eta_j^2}{16qM_q^{q-1}\delta_{j-1}} \ge \frac{C_1(q)N^{2-1/q}B^{q}J^2}{q4^{q+5}J^2B^{q}N^{1-1/q}2^{-2^{j-1}/N}} = \frac{C_1(q)N2^{2^{j-1}/N}}{q4^{q+5}}.
\ee
Using the inequality $e^x \ge 1+x > x$, we continue (\ref{q2})
$$
\ge \frac{C_1(q)}{q4^{q+5}}\frac{\ln 2}{2} 2^{j} \ge 2^{j+1+c}
$$
provided $C_1(q) \ge q4^{q+6+c/2}(\ln 2)^{-1}$. This implies
\be\label{2.12}
2N_j\exp\left(-\frac{m\eta_j^2}{16qM_q^{q-1}\de_{j-1}}\right)\le 2\exp(-2^{j+c}).
\ee

We now estimate $\mu^m\{\bz:\sup_{f\in {\mathcal N}_1}|L^q_\bz(f)|\ge\eta_1\}$.  We use Lemma \ref{AL1} with $g_j(\bz) = |f(\bx^j)|^q-\|f\|_q^q$. To estimate $\|g_j\|_\infty$ we use the Nikol'skii inequality (\ref{2.3}) and obtain for $f\in W$
$$
\|g_j\|_\infty \le (4BN^{1/q})^q + 1/2\le (4B)^qN+ 1/2.
$$
 Then Lemma \ref{AL1} gives for $m\ge C_1(q)N^{2-1/q}B^q J^2$
$$
\mu^m\{\bz:\sup_{f\in {\mathcal N}_1}|L^1_\bz(f)|\ge\eta_1\}\le 2N_1\exp\left(-\frac{m\eta_1^2}{4^{q+2}NB^q}\right) \le 1/4
$$
for sufficiently large $C_1(q)\ge C_04^{q+c/2}$ with absolute constant $C_0$. Substituting the above estimates into (\ref{2.10}), we obtain for large enough absolute constant $c$
$$
\mu^m\{\bz:\sup_{f\in W}|L^q_\bz(f)|\ge1/4\} <1.
$$
Therefore, there exists $\bz_0=(\xi^1,\dots,\xi^m)$ such that for any $f\in W$ we have
$$
|L^q_{\bz_0}(f)| \le 1/4.
$$
Taking into account that $\|f\|_q^q=1/2$ for $f\in W$ we obtain statement of Theorem \ref{T2.1}.

{\bf 2. Nikol'skii inequalities assumption.}

In this subsection we assume that a subspace $X_N$ satisfies
$NI(q,\infty)BN^{1/q}$ for all $q\in [2,\infty)$.
To obtain discretization results
we will use generic chaining argument and Dudley's entropy bound
combined with the
Gine--Zinn-type symmetrization argument.
The main result of this subsection is the following theorem.

\begin{Theorem}\label{DeDi2}
Let $q\in[2,+\infty)$
and assume that $X_N\in NI(q,\infty)H$.
Then for every $m\ge c\delta^{-2} q^2H^qN$
(here $c>0$ is a universal numerical constant)
$$
\bP\Bigl(\frac{1}{2}\|f\|_q^q\le \frac{1}{m}\sum\limits_{j=1}^m|f(\xi^j)|^q\le \frac{3}{2}\|f\|_q^q
\quad \forall f\in X_N\Bigr)\ge 1-\delta
$$
where $\xi^1, \ldots, \xi^m$ are independent random points with the distribution $\mu$.
In particular,
there are $m$ points $\xi^1, \ldots, \xi^m$ such that
$$
\frac{1}{2}\|f\|_q^q\le \frac{1}{m}\sum\limits_{j=1}^m|f(\xi^j)|^q\le \frac{3}{2}\|f\|_q^q
\quad \forall f\in X_N.
$$
\end{Theorem}

We firstly formulate several known lemmas that will
be used in the proof of this result and
we start with the concentration inequality for the
linear combinations of Bernoulli random variables.

\begin{Lemma}[\cite{LedTal}, Lemma 1.5]\label{C-lem}
\
Let $\chi_1,\ldots, \chi_m$ be independent symmetric Bernoulli random variables with values $\pm1$.
Then
$$
\bP\Bigl(\bigl|\sum\limits_{j=1}^m\chi_j\alpha_j\bigr|\ge
\bigl(\sum\limits_{j=1}^m|\alpha_j|^2\bigr)^{1/2}t\Bigr)\le 2e^{-t^2/2}\quad \forall t>0.
$$
\end{Lemma}

In this subsection
we need the entropy numbers in metric space
$(F, d)$.
However, it is more convenient to use the following
quantities
$$
e_n(A, d):=
\inf\Bigl\{\varepsilon\colon \exists f_1,\ldots, f_{N_n}\in A\colon
A\subset \bigcup\limits_{j=1}^{N_n}B_F(f_j, \varepsilon)\Bigr\},
$$
where $N_n=2^{2^n}$ for $n\ge 1$ and $N_0=1$
and where
$B_F(f, \varepsilon):=\{g\colon d(f,g)<\varepsilon\}$.
In other words, for a Banach space $X$
$$
e_n(A, d_X)  :=\e_{2^n}(A,X)
$$
where $d_X(f, g) = \|f-g\|_X$.

The following lemma is the Dudley's
entropy bound.

\begin{Lemma}[\cite{Tal}, Proposition 2.2.10
and the discussion before it]\label{T-lem}
Let $(F, d)$ be a metric space.
Assume that for a random process $\{W_f\}_{f\in F}$
for every pair $f, g\in F$ one has
$$
\bP\Bigl(|W_f - W_g|\ge t d(f, g)\Bigr)\le
2e^{-t^2/2}\quad \forall t>0.
$$
There is a numerical constant $C>0$ such that
for any $f_0\in F$ one has
$$
\mathbb{E}\bigl[\sup\limits_{f\in F}|W_f - W_{f_0}|\bigr]
\le
C\sum\limits_{n\ge 0} 2^{n/2}e_n(F, d).
$$
\end{Lemma}

Finally, we need the following observation based on
the Gine--Zinn-type symmetrization argument (see \cite[proof of Theorem 3]{GR07},
\cite[Lemma 3.1]{Kos}, \cite[Lemma 9.1.11]{Tal}).
This lemma reduces the study of discretization
to the study of the bounds for the expectation
of the supremum of a certain Bernoulli random process.

\begin{Lemma}\label{K-lem}
Assume that there is a number $\delta\in(0,1)$ such that,
for every fixed set of $m$ points $\{\bx_1,\ldots, \bx_m\}$,
one has
$$
\mathbb{E}_\varepsilon\sup\limits_{f\in B}
\Bigl|\sum\limits_{j=1}^m\chi_j|f(\bx_j)|^q\Bigr|\le
\Theta\sup\limits_{f\in B}
\Bigl(\sum\limits_{j=1}^m|f(\bx_j)|^q\Bigr)^{1-\delta}
$$
where $\chi_1,\ldots, \chi_m$ are independent symmetric
Bernoulli random variables with values $\pm1$.
Then for i.i.d. random vectors $\xi^j$ we have
\begin{multline*}
\mathbb{E}\Bigl[\sup\limits_{f\in B}
\Bigl|\frac{1}{m}\sum\limits_{j=1}^m|f(\xi^j)|^q - \|f\|_q^q\Bigr|
\Bigr]\le\\
2^{1/\delta}m^{-1}\Theta^{1/\delta}
+ 2\delta^{-1} \bigl(m^{-1}\Theta^{1/\delta}\bigr)^\delta
\Bigl(\sup\limits_{f\in B}\mathbb{E}|f(\xi^1)|^q\Bigr)^{1-\delta}.
\end{multline*}
\end{Lemma}

We provide a bound for the expectation
of the supremum of the Bernoulli random process from the previous lemma
in the following theorem.

\begin{Theorem}\label{DeDi1}
Let $\chi_1,\ldots, \chi_m$ be independent symmetric Bernoulli random variables with values $\pm1$.
Let $q\in[2,+\infty)$
and assume that $X_N\in NI(q,\infty)H$.
Then
$$
\mathbb{E}\Bigl[\sup\limits_{f\in X_N^q}
\Bigl|\sum_{j=1}^{m}\chi_j|f(\bx_j)|^q\Bigr|\Bigr]
\le cqH^{q/2}N^{1/2}\sup\limits_{h\in X_N^q}
\Bigl(\sum_{j=1}^{m}|h(\bx_j)|^q\Bigr)^{1/2}
$$
where $X_N^q:=\{f\in X_N\colon \|f\|_q\le 1\}$
and where $c>0$ is some universal numerical constant.
\end{Theorem}

\begin{proof}
We note that for any $f, g\in X_N^q$ one has
\begin{multline*}
\Bigl(\sum_{j=1}^{m}\bigl||f(\bx_j)|^q - |g(\bx_j)|^q\bigr|^2\Bigr)^{1/2}
\le\\
q\Bigl(\sum_{j=1}^{m}\bigl|\bigl(f(\bx_j) - g(\bx_j)\bigr)
\bigl(|f(\bx_j)|^{q-1} + |g(\bx_j)|^{q-1}\bigr)\bigr|^{2}\Bigr)^{1/2}
\le \\
2q\|f-g\|_\infty\sup\limits_{h\in X_N^q}
\Bigl(\sum_{j=1}^{m}|h(\bx_j)|^{2q-2}\Bigr)^{1/2}
\le \\
2q\|f-g\|_\infty\sup\limits_{h\in X_N^q}\Bigl[
\Bigl(\sum_{j=1}^{m}|h(\bx_j)|^q\Bigr)^{1/2}\|h\|_\infty^{q/2-1}\Bigr]
\le\\
2qH^{q/2-1}\|f-g\|_\infty\sup\limits_{h\in X_N^q}
\Bigl(\sum_{j=1}^{m}|h(\bx_j)|^q\Bigr)^{1/2}.
\end{multline*}
Let
$$
R:= 2qH^{q/2-1}\sup\limits_{h\in X_N^q}
\Bigl(\sum_{j=1}^{m}|h(\bx_j)|^q\Bigr)^{1/2}
$$
and let $W_f:= \sum_{j=1}^{m}\limits\chi_j|f(\bx_j)|^q$.
By Lemma \ref{C-lem}
with $\alpha_j = |f(\bx_j)|^q-|g(\bx_j)|^q$, we have
$$
\bP\Bigl(\bigl|W_f-W_g\bigr|\ge
tR\|f-g\|_q\Bigr)\le 2e^{-t^2/2}.
$$
By Lemma \ref{T-lem} we have
$$
\mathbb{E}\bigl[\sup\limits_{f\in X_N^q}|W_f - W_{0}|\bigr]
\le
CR\sum\limits_{n\ge 0} 2^{n/2}e_n(X_N^q, \|\cdot\|_\infty).
$$
We note that in our case $\|f-g\|_\infty\le H\|f-g\|_q$
implying that
$$
e_n(X_N^q, \|\cdot\|_\infty)\le H e_n(X_N^q, \|\cdot\|_q).
$$
We recall that $e_n(X_N^q, \|\cdot\|_q)
\le 3\cdot 2^{-2^n/N}$ (see \cite[Corollary 7.2.2]{VTbookMA}).
Thus,
\begin{multline*}
\sum\limits_{n\ge 0} 2^{n/2}e_n(X_N^q, \|\cdot\|_q)
\le \\
3\sum\limits_{n\ge 0} 2^{n/2}\cdot 2^{-2^n/N}
= 3\cdot 2^{-1/N} + 6 \sum\limits_{n=1}^\infty (2^n)^{-1/2}
\cdot 2^{-2^n/N}2^{n-1}
\le \\ \le
3 + 6 \sum\limits_{n=1}^\infty \int_{2^{n-1}}^{2^n}x^{-1/2}2^{-x/N}\, dx
\le
3+6\int_{0}^{\infty}x^{ - 1/2}2^{-x/N}\, dx =\\=
3 + 6 \bigl(\tfrac{N}{\ln 2}\bigr)^{1/2}
\int_{0}^{\infty}y^{- 1/2}e^{-y}\, dy =
3 + 6 \bigl(\tfrac{N}{\ln 2}\bigr)^{1/2}\Gamma(1/2).
\end{multline*}
Therefore, we get the bound
$$
\mathbb{E}\bigl[\sup\limits_{f\in X_N^q}|W_f - W_{0}|\bigr]
\le
C_1qH^{q/2}N^{1/2}\sup\limits_{h\in X_N^q}
\Bigl(\sum_{j=1}^{m}|h(\bx_j)|^q\Bigr)^{1/2}.
$$
The theorem is proved.
\end{proof}

We are ready to prove Theorem \ref{DeDi2}.

{\bf Proof of Theorem \ref{DeDi2}.}
From Lemma \ref{K-lem} we have that
$$
\mathbb{E}\Bigl[\sup\limits_{f\in X_N^q}
\Bigl|\frac{1}{m}\sum\limits_{j=1}^m|f(\xi^j)|^q - \|f\|_q^q\Bigr|
\Bigr]\le
4A + 4 A^{1/2}
$$
where $A=m^{-1}cq^2H^qN$.
If we take $m$ such that $A\le \frac{\delta^2}{256}<\frac{\delta}{16}$, then
$A^{1/2}\le \frac{\delta}{16}$ and we get that $4A + 4 A^{1/2}\le \frac{\delta}{2}$.
In that case we have
\begin{multline*}
\bP\Bigl(\sup\limits_{f\in X_N^q}
\Bigl|\frac{1}{m}\sum\limits_{j=1}^m|f(\xi^j)|^q - \|f\|_q^q\Bigr|\ge 1/2\Bigr)
\le\\ 2\mathbb{E}\Bigl[\sup\limits_{f\in X_N^q}
\Bigl|\frac{1}{m}\sum\limits_{j=1}^m|f(\xi^j)|^q - \|f\|_q^q\Bigr|
\Bigr] \le \delta.
\end{multline*}
Thus,
$$
\bP\Bigl(\frac{1}{2}\|f\|_q^q\le \frac{1}{m}\sum\limits_{j=1}^m|f(\xi^j)|^q\le \frac{3}{2}\|f\|_q^q
\quad \forall f\in X_N\Bigr)\ge 1-\delta.
$$
We now note that the bound $A\le \frac{\delta^2}{256}$ is equivalent to the inequality
$m\ge 256 c\delta^{-2} q^2H^qN$.
The corollary is proved.
\qed

Theorem \ref{DeDi2} can now be used to provide
new results for the discretization of the uniform norm.

\begin{Corollary}
Let $B, k>0$.
There is a positive number $c>0$ such that,
for every subspace
$X_N\in NI(q,\infty)BN^{k/q}$ $\forall q\in[2,+\infty)$,
for every $m\ge \frac{c}{[\log B]^2}N^{k+2}[\log N]^2$
there are $m$ points $\xi^1, \ldots, \xi^m$ such that
$$
\|f\|_\infty\le 2B^{k+1}\max\limits_{1\le j\le m}|f(\xi^j)|\quad \forall f\in X_N.
$$
\end{Corollary}

\begin{proof}
By Theorem \ref{DeDi2}, for every $q\in [2, +\infty)$ and for every number
$m\ge cq^2\bigl[BN^{k/q}\bigr]^{q}N$,
there are points $x_1, \ldots, x_m$ such that
\begin{multline*}
\|f\|_\infty\le BN^{k/q}\|f\|_p\le\\
BN^{k/q}2^{1/q}\Bigl(\frac{1}{m}\sum\limits_{j=1}^m|f(\xi^j)|^q\Bigr)^{1/q}
\le
2BN^{k/q}\max\limits_{1\le j\le m}|f(x_j)|.
\end{multline*}
We now take $q=\frac{\log N}{\log B}$. Then $B^q = N$ and
$$
q^2\bigl[BN^{k/q}\bigr]^{q}N = [\log B]^{-2} N^{k+2} [\log N]^2.
$$
The corollary is proved.
\end{proof}

\begin{Example}
{\rm
Let
$$
\Gamma(N):=\Bigl\{\mathbf{k}\in\Z^d:\  \  \   \prod_{j=1}^d \max\{ |k_j|, 1\} \leq N\Bigr\}
$$
It is known (see, for instance, \cite{DPTT}, Section 5.3)
that there exist two positive constants $C_1(d)$ and $C_2(d)$ such that
there are $m\le C_1(d)N[\log N]^{d-1}$ points $\xi^1, \ldots, \xi^m$
(actually forming a Smolyak net) for which
$$
\|f\|_\infty \le C_2(d)(\log N)^{d-1}\max\limits_{1\le j\le m}|f(\xi^j)|\quad \forall f\in \mathcal{T}(\Gamma(N)),
$$
where the subspace $\mathcal{T}(\Gamma(N))$ is defined in \eqref{eq-trig}.
In \cite{DPTT},
for each $d\in\N$ and each $N\in\N$, a set $W(N,d)\subset[0, 2\pi)^d$
was constructed with the following properties:
$|W(N, d)|\le C_1(d) N^{\alpha_d}(\log N)^{\beta_d}$,
$\alpha_d=\sum_{j=1}^d \frac{1}{j}$,
$\beta_d=d-\al_d$, and
$$
\|f\|_\infty \le C_2(d) \max_{\xi\in W(N,d) }|f(\xi)|\quad
f\in\mathcal{T}(\Gamma(N)).
$$
We now discuss what estimates Theorem \ref{DeDi2} provides for a random choice of points.
It is known (see Theorem 4.3.16 in \cite{VTbookMA}) that
$$
\mathcal{T}(\Gamma(N))\in NI(q,\infty)c(d)N^{1/q}[\log N]^{(d-1)(1-1/q)}.
$$
We also note that $\dim(\mathcal{T}(\Gamma(N))) \asymp N[\log N]^{d-1}$.
By Theorem \ref{DeDi2}, for
every $m\ge c\delta^{-2} q^2N^2[c(d)]^q[\log N]^{(d-1)q}$,
$$
\bP\Bigl(\frac{1}{2}\|f\|_q^q\le \frac{1}{m}\sum\limits_{j=1}^m|f(\xi^j)|^q\le \frac{3}{2}\|f\|_q^q
\quad \forall f\in X_N\Bigr)\ge 1-\delta.
$$
Let $\alpha>0$. We now take $q=\frac{\log N - (d-1)\log\log N}{\alpha(d-1)\log\log N+\alpha \log c(d)}$.
Then
$$
N^{1/q}[\log N]^{(d-1)(1-1/q)} = 2^{(\log N - (d-1)\log\log N)/q}[\log N]^{d-1}
= [c(d)]^\alpha[\log N]^{(\alpha+1)(d-1)}
$$
and
\begin{multline*}
q^2N^2[c(d)]^\alpha[\log N]^{(d-1)q}
\le \Bigl[\frac{\log N}{\alpha(d-1)\log\log N}\Bigr]^2 N^2 \Bigl[\frac{N}{[\log N]^{d-1}}\Bigr]^{1/\alpha}
= \\
\alpha^{-2}(d-1)^{-2} N^{2+1/\alpha}[\log N]^{2 - (d-1)/\alpha}[\log\log N]^{-2}.
\end{multline*}
Thus, for any $m\ge C_1(d)\delta^{-2}\alpha^{-2}N^{2+1/\alpha}[\log N]^{2 - (d-1)/\alpha}[\log\log N]^{-2}$,
we have
$$
\bP\Bigl(\|f\|_\infty\le C_2(d, \alpha)(\log N)^{(1+\alpha)(d-1)}\max\limits_{1\le j\le m}|f(\xi^j)|
\quad \forall f\in \mathcal{T}(\Gamma(N))\Bigr)\ge 1-\delta
$$
}
\end{Example}

In conclusion of this subsection we present another proof for the
discretization result under the entropy numbers condition.

\begin{Theorem}
Let $q\in(2,+\infty)$
and assume that $X_N$
satisfies the condition
$$
e_k(X^q_N,L_\infty) \le  B\left\{\begin{array}{ll}  (N/2^k)^{1/q}, &\quad k\le \log N,\\
 2^{-2^k/N},&\quad k\ge \log N,\end{array} \right.
$$
Then for every
$m\ge cq^2B^{q}N^{1-2/q}
\max\bigl\{\bigl(\frac{N^{1/2} - N^{1/q}}{1/2-1/q}\bigr)^2, N\bigr\}
$
(here $c>0$ is a universal numerical constant)
there are $m$ points $\xi^1, \ldots, \xi^m$ such that
$$
\frac{1}{2}\|f\|_q^q\le \frac{1}{m}\sum\limits_{j=1}^m|f(\xi^j)|^q\le \frac{3}{2}\|f\|_q^q
\quad \forall f\in X_N.
$$
\end{Theorem}

\begin{proof}
The assumptions of the theorem implies that
$X_N\in NI(q,\infty)BN^{1/q}$.
From the proof of Theorem \ref{DeDi1} we get the bound
$$
\mathbb{E}\Bigl[\sup\limits_{f\in X_N^q}
\Bigl|\sum_{j=1}^{m}\chi_j|f(\bx_j)|^q\Bigr|\Bigr]
\le
CR\sum\limits_{n\ge 0} 2^{n/2}e_n(X_N^q, \|\cdot\|_\infty)
$$
where
$$
R:= 2q(BN^{1/q})^{q/2-1}\sup\limits_{h\in X_N^q}
\Bigl(\sum_{j=1}^{m}|h(\bx_j)|^q\Bigr)^{1/2}
$$
By the assumptions of the theorem one has
\begin{multline*}
\sum\limits_{n\ge 0} 2^{n/2}e_n(X_N^q, \|\cdot\|_\infty)
\le
BN^{1/q}\sum\limits_{n\le \log N} 2^{n/2}2^{-n/q}
+B\sum\limits_{n\ge \log N} 2^{n/2}2^{-2^n/N}
\\\le
BN^{1/q}\frac{(2N)^{\frac{1}{2}-\frac{1}{q}}-1}
{2^{\frac{1}{2}-\frac{1}{q}}-1}+
2B \sum\limits_{n\colon 2^n\ge N} (2^n)^{-1/2}
\cdot 2^{-2^n/N}2^{n-1}
\\ \le
BN^{1/q}\frac{(2N)^{\frac{1}{2}-\frac{1}{q}}-1}
{\ln 2(1/2-1/q)}
+ 2B\sum\limits_{n\colon 2^n\ge N}
\int_{2^{n-1}}^{2^n}x^{-1/2}2^{-x/N}\, dx
\\
\le
BN^{1/q}\frac{(2N)^{\frac{1}{2}-\frac{1}{q}}-2^{\frac{1}{2}-\frac{1}{q}}
+N^{\frac{1}{2}-\frac{1}{q}}-1}
{\ln 2(1/2-1/q)}
+ 2B\int_{0}^{\infty}x^{ - 1/2}2^{-x/N}\, dx
\\ \le
\frac{3}{\ln 2}
 B\frac{N^{1/2}- N^{1/q}}{1/2-1/q}
+ 2B\bigl(\tfrac{N}{\ln 2}\bigr)^{1/2}\Gamma(1/2).
\end{multline*}
Thus, as in the proof of Theorem \ref{DeDi1},
by Lemmas \ref{C-lem} and \ref{T-lem}, we have
\begin{multline*}
\mathbb{E}\Bigl[\sup\limits_{f\in X_N^q}
\Bigl|\sum_{j=1}^{m}\chi_j|f(\bx_j)|^q\Bigr|\Bigr]
\\
\le
C_1
qB^{q/2}N^{1/2-1/q}\max\Bigl\{\frac{N^{1/2} - N^{1/q}}{1/2-1/q}, N^{1/2}\Bigr\}
\sup\limits_{h\in X_N^q}
\Bigl(\sum_{j=1}^{m}|h(\bx_j)|^q\Bigr)^{1/2}.
\end{multline*}
Now, by Lemma \ref{K-lem} we have
$$
\mathbb{E}\Bigl[\sup\limits_{f\in X_N^q}
\Bigl|\frac{1}{m}\sum\limits_{j=1}^m|f(\xi^j)|^q - \|f\|_q^q\Bigr|
\Bigr]\le 4A+4A^{1/2}
$$
where
$$
A=
m^{-1}
C_1^2
q^2B^{q}N^{1-2/q}\max\Bigl\{\Bigl(\frac{N^{1/2} - N^{1/q}}{1/2-1/q}\Bigr)^2, N\Bigr\}.
$$
Arguing now as in the proof of Theorem \ref{DeDi2},
we get the announced result.
The theorem is proved.
\end{proof}

\section{Remez-type inequalities for the multivariate trigonometric polynomials}
\label{Re}

In this section we consider the problem of estimating the
global norm $\|f\|_{L_\infty(\Omega)}$ in terms of the local norms
$\|f\|_{L_\infty(\Omega\setminus B)}$
when the measure of the set $B\subset \Omega$ is small.
One of the first results concerning this problem is the Remez inequality.
The classical version of this inequality (see \cite{re})
provides a sharp upper bound for the norm
$\|P_n\|_{L_\infty[-1,1]}$ of an algebraic polynomial $P_n$
when the measure of the subset of
$[- 1,1]\setminus B:=\{\bx\in[-1, 1]\colon |p_n(x)|\le 1\}$ is known.
A sharp multidimensional Remez-type inequality
for algebraic polynomials was obtained in \cite{br-ga}.

In the case of the univariate trigonometric polynomials $f\in \Tr([-n,n])$
the Remez inequality reads as follows: for any
Lebesgue measurable set $B\subset \mathbb{T}$ we have
\begin{equation}\label{rem-or}
\|f\|_{L_\infty(\T)}\le C(n,|B|)
\|f\|_{L_\infty(\T \setminus B)}.
\end{equation}
In \cite{er}, (\ref{rem-or}) was proved with $C(n,|B|)=\exp({4n|B|})$ for
$|B|< \pi/ 2$
(further historic discussion of this inequality can be found in \cite[Ch. 5]{bor}, \cite[Sec. 3]{ga}, and \cite{lor-g}).
The Remez-type inequality for general exponential sums
was obtained in \cite{Naz93}.

Multidimensional variants of Remez' inequality  for trigonometric polynomials $f\in \Tr([-n_1,n_1]\times\cdots\times[-n_d,n_d])$
were obtained in \cite{nurs}:
\be\label{Re2}
\|f\|_{L_\infty(\mathbb{T}^d)}\le
\exp\Big( 2d \big( |B|\prod_{j=1}^d{ n_j}\big)^{1/d} \Big)
\,
\|f\|_{L_\infty(\mathbb{T}^d \setminus B)}
\ee
for
$$
 |B|<
\Big(\frac{\pi}2\Big)^d  \frac{\Big(\min\limits_{1\leq j\leq d}{n_j}\Big)^d}{\prod_{j=1}^dn_j}.
$$
This improves the previous results for
the case of $n_1=\cdots=n_d$
from \cite{prymak} and \cite{kroo}.
Taking into account the fact $D(\bn):=\dim \Tr([-n_1,n_1]\times\cdots\times[-n_d,n_d])= \prod_{i=1}^d (2n_i+1)$, $\bn:=(n_1,\dots,n_d)$, we see that (\ref{Re2}) gives the inequality
$$
\|f\|_{L_\infty(\mathbb{T}^d)}\le
C(c,d)
\,
\|f\|_{L_\infty(\mathbb{T}^d \setminus B)}
$$
for  a subsets satisfying $|B|\le cD(\bn)^{-1}$.

We refer the reader to the paper \cite{TT} for further discussion of the Remez-type inequalities,
their applications, and connections to other inequalities. In this section we use one particular result from this paper.

 \begin{Theorem}\label{ReT1} Let $f$ be a continuous periodic function on $\T^d$. Assume that there exists a set $X_m=\{\bx^j\}_{j=1}^m \subset \T^d$ such that for all functions $f_\by(\bx):= f(\bx-\by)$ we have the discretization inequality
\be\label{Re3}
\|f_\by\|_\infty \le D\max_{1\le j\le m}|f_\by(\bx^j)|.
\ee
Then for any $B$ with $|B|<1/m$ we have
$$
\|f\|_{L_\infty(\Omega)} \le D\|f\|_{L_\infty(\Omega\setminus B)}.
$$
\end{Theorem}

We now
combine Theorem \ref{ReT1} and results on the discretization of the uniform norm
to provide new Remez-type inequalities for subspaces $\Tr(Q)$ with rather general $Q$.
Our first observation is a
corollary of Theorem \ref{AT2}.
Clearly, $\Tr(Q) \in NI(2,\infty)H$ with $H=|Q|^{1/2}$.
Therefore, Theorem \ref{AT2} and Theorem \ref{ReT1} imply the following Remez-type inequality.

\begin{Theorem}\label{ReT7} There are two absolute constants $c_1$ and $C_1$ such that for  any $Q\subset \Z^d$, for any $B\subset \T^d$, $|B|\le c_1|Q|^{-1}$, and for any $f\in \Tr(Q)$ (see \eqref{eq-trig}) we have
$$
\|f\|_{L_\infty(\T^d)} \le C_1|Q|^{1/2}\|f\|_{L_\infty(\T^d\setminus B)}.
$$
\end{Theorem}

\begin{proof}
We note that, for any $f\in \Tr(Q)$, one has $\|f_\by\|_\infty = \|f\|_\infty$ and
$$\max_{1\le j\le m}|f_\by(\bx^j)| = \max_{1\le j\le m}|f(\bx^j)|.$$
Theorem \ref{AT2} implies the estimate \eqref{Re3} with $D=C_1\sqrt{|Q|}$ and $m\le C_2|Q|$.
Applying Theorem \ref{ReT1} we get the announced result.
\end{proof}

We now show that one can get a sharper constant in the inequality,
but the tradeoff is in the measure of the admissible sets.
The following discretization result was obtained in \cite{KKT} (see a combination of Theorem 3.1 and Proposition 3.1).

\begin{Theorem}[\cite{KKT}]\label{ReT5}  For any $d\in\N$ there exists a positive constant $C(d)$ such that for any $Q\subset \Z^d$ we have, for any natural number $s\in [|Q|^{1/2},|Q|]$,
$$
D(Q,m,d) \le 6(e(1+|Q|/s))^{1/2} \quad \text{provided} \quad m\ge C(d)M|Q| (\log M)^3,
$$
where $D(Q,m,d)$ was defined in Remark \ref{AR3} and
$$
M:= [|Q|^2 e^{2s}(1+|Q|/s)^{2s}]+1.
$$
In a special case $s=|Q|$ we have a slightly better bound
$$
D(Q,m,d)  \le 12 \quad \text{provided} \quad m\ge C(d)M|Q| (\log M)^3,
$$
with
$$
M= 2^{4|Q|}.
$$
\end{Theorem}

\begin{Remark}\label{ReR1} The use of Theorem \ref{LiT2} in the proof of Theorem \ref{ReT5} from
\cite{KKT} allows us to replace $(\log M)^3$ by $\log M$ in Theorem \ref{ReT5}.
\end{Remark}

Theorem \ref{ReT5} with $s=|Q|$, Remark \ref{ReR1}, and Theorem \ref{ReT1} imply the following Remez-type inequality
for $\Tr(Q)$ with any $Q\subset \Z^d$.

\begin{Theorem}\label{ReT6} For any $d\in \N$ there exists a positive constant $c(d)$ such that for  any $Q\subset \Z^d$, for any $B\subset \T^d$, $|B|\le c(d)2^{-4|Q|}|Q|^{-2}$, and for any $f\in \Tr(Q)$ (see \eqref{eq-trig}) we have
$$
\|f\|_{L_\infty(\T^d)} \le 12\|f\|_{L_\infty(\T^d\setminus B)}.
$$
\end{Theorem}

We point out that Theorem \ref{ReT6} holds for all subsets $Q$ with the constant $12$ in the Remez inequality. However, we are paying a big price for that -- we can only delete a set $B$ of
exponentially small measure.

\section{Some comments}
\label{C}

In Section \ref{A} we introduced some results,
which guarantee for an arbitrary $N$-dimensional subspace $X_N\subset \C(\Omega)$ existence of a set of $m \le C_1N$ points with good properties, where $C_1$ is an absolute constant. It is clear that $C_1\ge 1$. In this subsection we comment on the case, when $C_1$
is close to $1$.

We begin with comments on the case of weighted discretization, namely, on Theorem \ref{AT1}. It is pointed out in \cite{VT159} that the paper by J. Batson, D.A. Spielman, and N. Srivastava \cite{BSS}  basically solves the discretization problem with weights.  We present an explicit formulation of this important result in our notation.

\begin{Theorem}[{\cite[Theorem 3.1]{BSS}}]\label{BSS} Let  $\Omega_M=\{x^j\}_{j=1}^M$ be a discrete set with the probability measure $\mu_M(x^j)=1/M$, $j=1,\dots,M$, and let $X_N$ be an $N$-dimensional subspace of real functions defined on $\Omega_M$.
Then for any
number $b>1$ there  exists a set of weights $\lambda_j\ge 0$ such that $|\{j: \lambda_j\neq 0\}| \le \lceil bN \rceil$ so that for any $f\in X_N $ we have
\begin{equation*}
\|f\|_2^2 \le \sum_{j=1}^M \lambda_jf(x^j)^2 \le \frac{b+1+2\sqrt{b}}{b+1-2\sqrt{b}}\|f\|_2^2.
\end{equation*}
\end{Theorem}

As observed  in \cite [Theorem 2.13]{DPTT},  this last theorem with  a general probability space $(\Omega,  \mu)$ in place of the discrete space  $(\Omega_M, \mu_M)$ remains true (with other constant in the right hand side) if $X_N\subset L_4(\Omega,\mu)$.  It was proved in \cite{DPSTT2}  that the additional assumption $X_N\subset L_4(\Omega,\mu)$ can be dropped.

\begin{Theorem}[{\cite[Theorem 6.3]{DPSTT2}}]\label{DPSTT} If $X_N$ is an $N$-dimensional subspace of the real $L_2(\Omega,\mu)$, then for any $b\in (1,2]$, there exist a set of $m\leq    \lceil bN \rceil$ points $\xi^1,\ldots, \xi^m\in\Omega$ and a set of nonnegative  weights $\lambda_j$, $j=1,\ldots, m$,  such that
\[ \|f\|_2^2\leq  \sum_{j=1}^m \lambda_j f(\xi^j)^2 \leq  \frac C{(b-1)^2}  \|f\|_2^2,\  \ \forall f\in X_N,\]
where $C>1$ is an absolute constant.
\end{Theorem}

Replacing in the proof of Theorems \ref{AT2} and \ref{AT3} the use of Theorem \ref{AT1} by the use of Theorem \ref{DPSTT} we obtain the statement of Remark \ref{AR2} below  in the case of subspaces of real functions.   Very recent results from
\cite{BSU}, where Theorem \ref{BSS} was extended to the complex case, allow us to obtain Remark
\ref{CR2} in the complex case as well.

\begin{Remark}\label{CR2}   Theorems \ref{AT2} and \ref{AT3}  hold
with $C_1=b$, where $b$ is any number greater than $1$ and with $C_2=C(b)$ being a constant, which
is allowed to depend on $b$.
\end{Remark}

We now compare Theorems \ref{AT2} and \ref{LiT3}.
Both of them have an extra factors
in the discretization inequalities.
In Theorem \ref{AT2} that factor comes from the Nikol'skii
inequality for the pair $(2,\infty)$ and in
Theorem \ref{LiT3} it is a special bilinear approximation
characteristic of the corresponding Dirichlet kernel.
We now show that these
extra factors are related.

As we have already mentioned in Remark \ref{AR2},
one has $(2,\infty)$ Nikol'skii
inequality \eqref{I4a} with the factor
$$
H = \max_\bx
\left(\sum_{j=1}^N u_j(\bx)^2\right)^{1/2}
=\max_\bx (\cD(X_N,\bx,\bx))^{1/2}.
$$
Thus, for any $M$, we have
$$
\sigma_M^c(\D(X_N))_{(\infty,1)} \le \max_\bx \|\D(X_N,\bx,\cdot)\|_1
$$
$$
\le \max_\bx \|\D(X_N,\bx,\cdot)\|_2 =  \max_\bx (\cD(X_N,\bx,\bx))^{1/2}.
$$

Therefore, the extra factor in Theorem \ref{LiT3} is always better than in Theorem \ref{AT2}.
However, the condition on $m$ in Theorem \ref{AT2} is much better than in Theorem \ref{LiT3}.

\vskip .2in

{\bf Acknowledgements.}
The authors would like to thank the anonymous referee
for the helpful comments and corrections which have allowed to significantly improve the paper.

The authors are thankful to David Krieg for bringing the paper \cite{KW} to their attention.

E. Kosov and results of Sections 2 and 4 (subsection 2)
were supported by the Russian Science Foundation (project No. 22-11-00015)
at the Lomonosov Moscow State University.

V. Temlyakov and results of Sections 3, 4 (subsection 1a and 1b), 5, 6
were supported by the Russian Science Foundation (project No. 23-71-30001)
at the Lomonosov Moscow State University.

{\bf Competing interests.}
The authors have no competing interests to declare that are relevant to the content of this article.

\Addresses

\end{document}